\documentclass{article}
\usepackage[margin=1.5in]{geometry}
\usepackage[british]{babel}
\usepackage{amsthm}
\usepackage{mathtools}
\usepackage{amssymb, enumerate}
\usepackage{latexsym}
\usepackage{url}
\usepackage[all,cmtip]{xy}
\usepackage{amsmath,amssymb}
\usepackage{tikz-cd}
\usepackage{faktor}

\newcounter{itemcounter}
\numberwithin{itemcounter}{section}

\newtheorem{thm}[itemcounter]{Theorem}

\newtheorem{prop}[itemcounter]{Proposition}
\newtheorem{cor}[itemcounter]{Corollary}

\newtheorem{rem}[itemcounter]{Remark}

\newtheorem*{thm*}{Theorem}
\newtheorem*{con*}{Conjecture}
\newtheorem*{cor*}{Corollary}
\newtheorem*{ack*}{Acknowledgements}

\newcommand{\Sym}{\mathop{\rm Sym}\nolimits}
\newcommand{\Syl}{\mathop{\rm Syl}\nolimits}
\newcommand{\Ker}{\mathop{\rm Ker}\nolimits}

\newcommand{\Hom}{\mathop{\rm Hom}\nolimits}
\newcommand{\sHom}{\mathop{\rm \underline{Hom}}\nolimits}

\newcommand{\Aut}{\mathop{\rm Aut}\nolimits}
\newcommand{\Ext}{\mathop{\rm Ext}\nolimits}
\newcommand{\Mod}{\mathop{\rm Mod}\nolimits}
\newcommand{\sMod}{\mathop{\rm \underline{Mod}}\nolimits}

\newcommand{\Out}{\mathop{\rm Out}\nolimits}

\newcommand{\Pic}{\mathop{\rm Pic}\nolimits}
\newcommand{\sPic}{\mathop{\rm \underline{Pic}}\nolimits}

\newcommand{\cT} {\mathcal{T}}
\newcommand{\cL} {\mathcal{L}}
\newcommand{\cE} {\mathcal{E}}
\newcommand{\cO} {\mathcal{O}}

\title{Picard groups for some blocks with TI defect groups}
\date{}
\author{Claudio Marchi\footnote{Department of Mathematics, School of Natural Sciences, University of Manchester, Manchester, M13 9PL, United Kingdom. Email: claudio.marchi@manchester.ac.uk}}

\begin{document}
\maketitle
\begin{abstract}
We calculate the Picard groups for principal blocks $B$ with TI defect groups and cyclic inertial quotient. The methods used generalize results in \cite{carlson2000self} on self stable equivalences and take advantage of the existence of equivalences given by Green correspondence in this setting. In particular, we show that $\Pic(B)=\cE(B)$, giving more evidence for a conjecture on basic auto-Morita equivalences.
\end{abstract}
\section{Introduction} 
Let $(K,\cO,k)$ be a $p$-modular system, meaning that $k$ is an algebraic closure of the field with $p$ elements, $\cO$ is a complete discrete valuation ring, $\cO/J(\cO)\simeq k$ and $K$ is the field of fractions of $\cO$, with $\text{char}(K)=0$. We also assume that $K$ is a splitting field for all the groups involved in this paper. By a block, we will always mean a block algebra of $\cO G$.

The Picard group of a block, i.e. the group of auto-Morita equivalences of a block, has recently been object of many investigations, revealing itself a particularly interesting Morita invariant. One of the main questions regarding Picard groups of blocks is about $\cE(B)$, the subgroup of $\Pic(B)$ given by bimodules with endo-permutation source as $\cO(G\times G)$-modules. As reported in \cite{eisele2020picard}, it is in fact an open question whether this subgroup is proper. We state it as a conjecture here.
\begin{con*}
Let $B$ be a block of a finite group $G$. Then $\Pic(B)=\cE(B)$.
\end{con*}
This conjecture holds for nilpotent blocks and blocks with cyclic defect group by classical results in \cite{puig1988nilpotent} and \cite{linckelmann1996isomorphism}, and it has almost been settled for blocks with normal defect groups in \cite{livesey2021picard} and \cite{lima20}, but little is known for other classes of blocks. Principal blocks with (non-normal) trivial intersection defect groups, look like perfect candidates to test the conjecture, due to the existence of a stable equivalence of Morita type with a block with normal defect group. The main result of this paper gives a positive answer for some of these blocks. We recall, before stating it, that blocks with TI defect groups were completely classified in \cite{an2004blocks}, but information concerning principal blocks can be retrieved in \cite{blau1990modular}.
\begin{thm*}
Let $G$ be a group with (non-normal) T.I. Sylow $p$-subgroups. If the principal block $B$ has cyclic inertial quotient, then $\Pic(B)=\cE(B)$.
\end{thm*}
The main tools used for proving this result are generalizations of the results of \cite[Section 3]{carlson2000self} and recent developments in the study of Picard groups. In particular, replacing the condition on the freeness of the action of the inertial quotient with an hypothesis on endo-trivial modules in the results of Carlson and Rouquier, allows the description of the stable Picard group for a much larger set of groups. Stable equivalences of Morita type are quite common in block theory, e.g. derived equivalent blocks are stable equivalent à la Morita, therefore we hope this kind of approach could be used in future to describe Picard groups whenever there is a stable equivalence of Morita type with a block with normal defect group, so for example when Broué's abelian defect group conjecture holds.

\section{Stable auto equivalences of Morita type}
Recall that the stable category of $\Mod(A)$, where $A$ is a $k$-algebra, is the category $\sMod(A)$ whose objects are the $A$ modules and, for any two $A$-modules $U,V$, morphisms are given by the quotients space
\[
\sHom_A(U,V)=\Hom_A(U,V)/\Hom_A^{pr}(U,V),
\]
where $\Hom_A^{pr}(U,V)$ are the $A$-homomorphisms that factor through a projective $A$-module. In this section we are going to completely describe the stable Picard group for some group algebras, $\sPic(kG)$, whose elements are the isomorphism classes in the stable category of $kG$-$kG$-bimodules inducing a stable auto-equivalence of Morita type of $kG$. In general it is not known how this group is related to the Picard group of $kG$, but we will show that, in some cases, $\sPic(kG)$ is generated by $\Pic(kG)$ and a cyclic, central, subgroup.

As we said in the introduction, the aim of this paper is calculating the Picard group for some principal blocks with TI defect groups. The choice of such blocks restricts our attention to groups $G$ with TI Sylow $p$-subgroups $P$ and, by \cite[Proposition 5.2.5]{linckelmann2018block}, there is a stable equivalence of Morita type given by induction and restriction, that gives
\[
\Psi:\sMod(B)\xrightarrow{\sim}\sMod\left(k(N_G(P))\right)
\] 
The stable Picard groups of $B$ and $kN_G(P)$ are then isomorphic, so this is the ultimate reason for being interested in the stable Picard group of $N_G(P)$. The following is an adaptation of \cite[Lemma 3.1]{carlson2000self} to a more general setting. We are going to make use of endo-trivial modules and $T(G)$ will always denote the group of endo-trivial modules. A complete account of the topic can be found in \cite{mazza2019endotrivial}.
\begin{prop}\label{carlsonrouquier}
Let $G=P\rtimes E$ with $E$ a cyclic $p'$-group, $P$ a $p$-group with non-periodic cohomology. Suppose that $T(P)\simeq \mathbb{Z}$ corresponds to the Heller translates of the trivial $kP$-module. If $F:\sMod(kG)\rightarrow \sMod(kG)$ is a stable equivalence of Morita type, then there is $\sigma\in \Sym(\Hom(E,k^{\times}))$ and an integer $n$ such that $\forall V\in \Hom(E,k^{\times})$(identifying the characters of $E$ with the underlying modules) we have $\Omega^{-n}F(V)\simeq \sigma(V)$ in $\sMod(kG)$.
\end{prop}
\begin{proof}
Let $V\in\Hom(E,k^{\times})$, then $V$ is an endo-trivial $kG$-module, thus $F(V)$ is endo-trivial as well; this last fact immediately follows by \cite[Proposition 7.5.12]{linckelmann2018block} and associativity of tensor product. But then its restriction to $P$ is endo-trivial by \cite[Proposition 2.6]{carlson2006endotrivial}, thus there is an integer $n$ such that $F(V)$ is a direct summand of
\[
\text{Ind}_P^G(\Omega^n(k)+\text{proj})=\oplus_{U\in \hat{E}}\Omega^n_{kG}U+\text{proj}.
\]
So there are two functions, $n:\hat{E}\rightarrow \mathbb{Z}$, $\sigma:\hat{E}\rightarrow\hat{E}$, such that 
\[
F(V)\simeq \Omega^{n(V)}\sigma(V).
\]
We want to show that $\sigma$ is a permutation. Suppose $\sigma(V)=\sigma(V')$, then $V'=\Omega^{n(V)-n(V')}_{kG}(V')$, thus after restriction to $P$ we have $k=\Omega^{n(V)-n(V')}_{kP}(k)$, but this is a contradiction since $P$ is not periodic, so $\sigma$ is a permutation.

We can now define a bijection $\tau:\hat{E}\rightarrow\hat{E}$ by $\tau(V)=\sigma(VW)\sigma(W)^{-1}$. Take $W\in\hat{E}$ with $n(W)$ minimal and define $h:\hat{E}\rightarrow\mathbb{Z}[[t,t^{-1}]]$ by
\[
h(V)=\sum_{n\in\mathbb{Z}}\text{dim}(\Ext^n(k,V))t^n.
\]
Proceeding exactly as in \cite[Lemma 3.1]{carlson2000self} we get the desired result.
\end{proof}
\begin{cor}\label{stablepic}
Let $G$ be a finite group, $P$ a Sylow $p$-subgroup of $G$, $H:=N_G(P)\simeq P\rtimes E$ such that:
\begin{enumerate}[(i)]
\item $P$ has non-periodic cohomology.
\item $T(P)\simeq\mathbb{Z}$.
\item $E$ is a cyclic $p'$-group.
\end{enumerate} 
Then
\[
\underline{\textnormal{Pic}}(\mathcal{O}H)=\textnormal{Pic}(\mathcal{O}H)\cdot \langle \Omega_{\mathcal{O}(H\times H^0)}\mathcal{O}H\rangle.
\]
\end{cor}
\begin{proof}
We obviously have a map $\sPic(\cO G)\rightarrow\sPic(kG)$, and the Kernel of this map is contained in $\Pic(\cO G)$, so, considering that Heller translates of $\cO H$ are naturally contained in $\sPic(\cO H)$, we just need to prove the result over $k$. We want to prove then that
\[
\underline{\textnormal{Pic}}(kH)=\textnormal{Pic}(kH)\cdot \langle \Omega_{k(H\times H^0)}kH\rangle.
\]
Consider $M$ an invertible $kH$-$kH$-bimodule and $M\otimes_{kH}-$ the stable equivalence of Morita type induced on $kH$. By Proposition \ref{carlsonrouquier} there is an $n$ such that if $M_1=\Omega^{-n}M$, then $M_1\otimes_{kH}-$ permutes the simple $kH$-modules in the stable category, i.e. up to projectives. But then by \cite[Proposition 2.4]{linckelmann1996stable} we can consider (uniquely) an indecomposable summand $M_2$ of $M_1$ that actually permutes the simple $kH$-modules in the module category. This means, by \cite[Proposition 2.5]{linckelmann1996stable}, that $M_2\in\text{Pic}(kH)$.
\end{proof}
\section{Picard groups}
For the next results we need some notions on finite groups of Lie type and their Sylow subgroups that are available in the literature. In particular, most of the information needed can be found in \cite{wilson2009finite}. We will also widely use notions on Suzuki groups, that are mainly available in the original paper by Suzuki, \cite{suzuki1962class}.
\begin{prop}
For the following groups
\begin{enumerate}[(i)]
\item $G={}^{2}B_{2}(q)$, $P\in\Syl_2(G)$, $q=2^m$, $m=2n+1$;
\item $G={}^{2}G_{2}(q)$, $P\in\Syl_3(G)$, $q=3^m$, $m=2n+1$;
\item $G=PSU_3(q)$, $P\in\Syl_p(G)$, $q=p^m$;
\end{enumerate} 
$H=N_G(P)$ satisfies the hypothesis of Proposition \ref{carlsonrouquier}.
\end{prop}
\begin{proof}
(i) $P$ is a non-abelian $2$-group of order $2^{2(2n+1)}$, and is usually called in the literature a Suzuki $2$-group of type A. Detailed information on these groups can be found in \cite{higman1963suzuki} and obviously \cite{suzuki1962class}, but all we need to know is that $P$ is a special group with
\[
Z(P)=\Phi(P)=P'\simeq C_2^{2n+1},
\]
and all involutions of $P$ are central, so we conclude by \cite[Corollary 1.3]{carlson2005classification} that the group of endo-trivial $kP$-modules is $T(P)\simeq \mathbb{Z}$, generated by $\Omega^1(k)$. For what concerns $H:=N_G(P)$, we have that $H\simeq P\rtimes C_{2^{2n+1}-1}=P\rtimes \langle \xi\rangle$. 
\newline
(ii) In this case $H=P\rtimes E$, where $E$ is cyclic of order $q-1$ and the condition on endo-trivial modules follows from \cite[Theorem 5.6]{carlson2006endotrivial}.
\newline
(iii) This time $H=P\rtimes E$, where $E$ is a cyclic group of order $q^2-1$ and the statement on endo-trivial modules follows from \cite[Theorem 7.1]{carlson2020torsion}.
\end{proof}
\begin{rem}
We didn't need to explicitly check the hypothesis on the periodicity of the cohomology for $P$ because it immediately follows from \cite[Proposition 9.3]{brown2012cohomology}.
\end{rem}
We are now ready to compute the Picard groups. We start calculating $\Pic(B)$ for some blocks of finite simple groups, and then it will be proved in the main Theorem that these are essentially the only calculations we need to do for describing the Picard groups for principal blocks with TI defect groups and cyclic inertial quotient. 

The proof is essentially the same for all the simple groups considered, so we are going to put more detail in the first of these proof and the others will just be an adaptation of the same argument.
\begin{thm}\label{picardsuzuki}
Let $q:=2^m$, with $m$ an odd number greater than $2$, $G={}^{2}B_{2}(q)$, $H:=N_G(P)$, where $P$ is a Sylow $2$-subgroup of $G$, and $B$ be the principal block of $\cO G$. Then
\begin{itemize}
\item $\Pic(B)=\cT(B)\simeq C_m$
\item $\Pic(\cO H)=\cT(\cO H)\simeq C_{q-1}\rtimes C_{m}$
\end{itemize}
\end{thm}
\begin{proof}
We claim that $\sPic(\cO H)=\underline{\cE}(\cO H)$, where
\[
\underline{\cE}(\cO H)=\left\lbrace M\in \sPic(\cO H)\vert\, M\textnormal{ has endo-permutation source}\right\rbrace
\]
First we prove that $\Pic(\cO H)=\cT(\cO H)$, i.e. all auto-Morita equivalences of $\cO H$ are given by trivial source bimodules. Recall that 
\[
H\simeq P\rtimes C_{2^{q}-1}=:P\rtimes E,
\]
and obviously 
\[
\left\lbrace\chi\in \textnormal{Irr}(H)\vert P\leq \Ker(\chi)\right\rbrace=\left\lbrace \chi\in \textnormal{Irr}(H)\vert\chi\textnormal{ is a lift of } \psi\in \textnormal{IBr}(H)\right\rbrace.
\]
So, by \cite[Corollary 4.5]{eaton2020some}, it follows that $\Pic(\cO H)=\cT(\cO H)$.

Now by Corollary \ref{stablepic} we have that elements of $\sPic(\cO H)$ are either bimodules inducing a self-Morita equivalence of $\cO H$ or Heller translate of those. Since Heller translates of bimodules with trivial source have endo-trivial source, we can immediately conclude that $\sPic(\cO H)=\underline{\cE}(\cO H)$. 

We have already seen that the Sylow $2$-subgroups of $G$ are TI-groups, thus induction and restriction give a stable equivalence of Morita type. In particular there is a $\cO H$-$B$-bimodule $M$, with trivial source and diagonal vertex $\Delta(P)$, such that
\[
\Psi:M\otimes_B -\otimes_B M^{*}:\sPic(B)\rightarrow\sPic(\cO H)
\]
is a group homomorphism. We know that $H$ has non-periodic cohomology, thus $\langle \Omega_{\cO(H\times H^0)}\cO H\rangle$ is an infinite central cyclic subgroup of $\sPic(\mathcal{O}H)$. But then all finite subgroups of $\sPic(\mathcal{O}H)$ are contained in $\Pic(\cO H)$, and in particular this holds for the image of the map $\Psi$ as well. Thus we have an injective map $\Pic(B)\hookrightarrow \Pic(\cO H)$, that maps $\Out_P(A)$ to $\Hom(E,k^{\times})$. 

We already observed that $\textnormal{Pic}(\mathcal{O}H)=\mathcal{T}(\mathcal{O}H)$, and that $\Psi$ is induced by tensoring with trivial source bimodules, so $\textnormal{Pic}(B)=\mathcal{T}(B)$.

Using the notation introduced in \cite{boltje2020picard} we have an exact sequence
\[
\begin{tikzcd}
1 \arrow[r] & \textnormal{Out}_{P}(A) \arrow[r]           & \mathcal{T}(B) \arrow[r] & {\textnormal{Out}(P,\mathcal{F})} 
\end{tikzcd}
\]
Moreover, since the trivial character is the only lift of a Brauer character(just look at the character degrees), we must have that $\textnormal{Out}_P(A)=1$, otherwise a non-trivial element of $\textnormal{Out}_P(A)=1$ would correspond to a non-trivial element of $\Hom(E,k^{\times})$ fixing the character $1_H$, but multiplication by a non-trivial linear character does not fix the trivial character.

All we need to understand is then the "fusion part" of $\mathcal{T}(B)$. By \cite[Corollary 7.11]{craven2011theory}, we have that Sylow $2$-subgroups of $G$ are resistant $2$-groups, so 
\[
\textnormal{Out}(P,\mathcal{F})\simeq N_{Aut(P)}(E)/E.
\]
Now we need to study the structure of $\textnormal{Aut}(P)$. We follow the notation in \cite{lewis2014bounding}. Identify $P$ with the group $A(m,\Theta)$, where $\Theta$ is a non-trivial automorphism of $\mathbb{F}:=GF(q)$ and elements of $A(m,\Theta)$ are given by pairs $(a,b)$, $a,b\in\mathbb{F}$ and multiplication is defined by
\[
(a,c)(b,d)=(a+b,c+d+b\Theta(a)).
\]
In particular the center of $P$ consists of elements of the form $(0,b)$, $b\in\mathbb{F}$. We can then define three subgroups of $\Aut(P)$ given by
\begin{itemize}
\item $A_1=\left\lbrace \varphi_{\psi}\vert\psi\textnormal{ is a linear transformation of }\mathbb{F}\textnormal{ as a vector space}\right\rbrace$, where $(a,b)^{\varphi_{\psi}}=(a,\psi(a)+b)$;
\item $A_2=\left\lbrace \varphi_{x}\vert x\in\mathbb{F}^{\times}\right\rbrace=E$, where $(a,b)^{\varphi_{x}}=(xa,x\Theta(x)b)$;
\item $A_3=\left\lbrace \phi_{\sigma}\vert\sigma\in \mathcal{G}al(\mathbb{F}\vert \mathbb{F}_2)\right\rbrace$, where $(a,b)^{\phi_{\sigma}}=(a^{\sigma},b^{\sigma})$.
\end{itemize}
Lewis in \cite[Theorem 6.5]{lewis2014bounding} proves that $\textnormal{Aut}(P)=A_1A_2A_3$ and he already points out that $E$ is normalized by $A_3$, so we just need to check whether any element of $A_1$ normalizes $E$ or not. But
\[
\varphi^{\varphi_{\psi}}_x(a,b)=(xa,x\psi(a)+xb-\psi(xa)),
\]
thus no (non-trivial) element of $A_1$ normalizes $A_2$, so we can immediately conclude that $N_{\textnormal{Aut}(P)}(E)=E\rtimes A_3$, in particular 
\[
\mathcal{T}(B)\lesssim A_3\simeq C_{m}
\]
$A_3$ is however given by the restriction of a field automorphism of $G$ to its Sylow $2$-subgroup, thus, since $B$ is $\Out(G)$-stable, we can define, for each $\varphi\in A_3$, an automorphism $\hat{\varphi}$ of $G$ such that ${}_{\varphi}B$ is a non-trivial element of $\textnormal{Pic}(B)$. Then we conclude that
\[
\textnormal{Pic}(B)=\mathcal{T}(B)\simeq C_m, \textnormal{ and } \textnormal{Pic}(\mathcal{O}H)=\mathcal{T}(\mathcal{O}H)\simeq C_{q-1}\rtimes C_{m}
\]
\end{proof}
\begin{thm}\label{simplecases}
Let $G$ be one of the following groups:
\begin{enumerate}[(i)]
\item ${}^{2}G_{2}(q)$, $q=p^m$, $p=3$, $m$ odd;
\item $PSU_3(q)$, $q=p^m$;
\item $PSL_2(q)$, $q=p^m$, $p$ odd.
\end{enumerate} 
Take $H=N_G(P)$, where $P$ is a Sylow $p$-subgroup of $G$, and $B$ the principal $p$-block of $G$. Then, respectively,
\begin{enumerate}[(i)]
\item $\Pic(B)=\cT(B)\simeq C_m$, $\Pic(\cO H)\simeq C_{q-1}\rtimes C_{m}$;
\item $\Pic(B)=\cT(B)\simeq C_{(3,q+1)}\rtimes C_{2m}$, $\Pic(\cO H)\simeq C_{q^2-1}\rtimes\left(C_{(3,q+1)}\rtimes C_{2m}\right)$;
\item $\Pic(B)=\cT(B)\simeq C_{(2,q-1)}\rtimes C_{m}$, $\Pic(\cO H)\simeq C_{q-1}\rtimes\left(C_{(2,q-1)}\rtimes C_{m}\right)$.
\end{enumerate}
\end{thm}
\begin{proof}
(i) As before, we can prove that $\Pic(\cO H)=\cT(\cO H)$ just by looking at characters. In fact, the only characters with $P$ in their kernel are the linear characters of $H$ and thus the only lifts of Brauer characters of $H$. Then, by Corollary \ref{stablepic} again, $\sPic(\cO H)=\underline{\cE}(\cO H)$ and $\Pic(B)$ injects in $\Pic(\cO H)$, yielding $\Pic(B)=\cT(B)$.

We have again that $\Out_P(A)$ is trivial, so we just need to understand the fusion part. However, by \cite[Theorem A]{broto2019automorphisms}, these are just induced by an outer automorphism of the group $G$. Since $\Out(G)$ is cyclic, generated by the field automorphism, and $B$ is $\Out(G)$-stable, every field automorphism induces a non-trivial auto-Morita equivalence, we have that $\Pic(B)=\cT(B)\simeq C_q$.

(ii)The proof goes exactly like before.

(iii)This case is more immediate, since the Sylow $p$-subgroup of $G$ is abelian, $N_G(P)\simeq P\rtimes C_{q-1}$, and $C_{q-1}$ acts freely on $P$, thus the original result by Carlson-Rouquier can be applied. The proof is analogous to the one of \cite[Proposition 5.3]{eaton2020some}.
\end{proof}

\begin{thm}
If $B$ is a principal block with (non-normal) TI defect groups, and cyclic inertial quotient then $\Pic(B)=\cE(B)$.
\end{thm}
\begin{proof}
Let $B$ be the principal block of $G$. If $P$ has cyclic defect group then, by \cite[Theorem 2.7]{linckelmann1996isomorphism}, $\Pic(B)=\cE(B)$.

Suppose that $P$ is a generalized quaternion group. Then, by \cite[Lemma 1.1]{blau1990modular}, $B$ is the principal block of a $p$-nilpotent group for $\vert P\vert\geq 16$ and then $B$ is basic Morita equivalent to $\cO P$. Since $\Pic(\cO P)=\cL(P)$ by \cite{roggenkamp1987isomorphisms}, we have just to deal with $P\simeq Q_8$, $G/O_{2'}(G)\simeq SL(2,3)$. However, in this last case, $B$ is solvable and, by \cite[Lemma 5.9]{eaton2020donovan}, $G$ has $2$-length one. But then $P$ is normal in $G$, and we are not concerned with this case.

We can now assume that $C_p\times C_p\leq P$, ruling out the previous cases. By \cite[Theorem 6.2.4]{gorenstein2007finite}, we have that $O_{p'}(G)\leq N_{G}(P)$, thus $G/O_{p'}$ has trivial intersection Sylow $p$-subgroups and, since $B$ is the group ring $\cO G/O_{p'}(G)$, we can assume $O_{p'}(G)=1$. 
\newline
Take now $M$ a minimal normal subgroup of $G$. Obviously $M\simeq S_1\times...\times S_n$, where $S_i$ are simple groups, and we can assume, up to reordering, that $S_1$ is not abelian, otherwise $G$ would be solvable and, by \cite[Lemma 1.1]{blau1990modular}, we would be again in the case $P$ cyclic or generalized quaternion. Since $P\cap M$ is a trivial intersection Sylow $p$-subgroup of $M$ and there is $s_1\in S_1$ such that $(P\cap M)^{s_1}\neq P\cap M$, we must have that $M$ is a simple group, and then $S\leq G\leq \Aut(S)$. 

Since we are restricting our attention to blocks with cyclic inertial quotient, $S$ must be isomorphic to one of the following groups:
\begin{enumerate}[(i)]
\item $PSL_2(q)$
\item $PSU_3(q)$
\item ${}^2B_2(q)$
\item ${}^2G_2(q)$
\end{enumerate}
If $G$ is not simple and $S$ is a Suzuki group or a Ree group, then the inertial quotient of the principal block is not cyclic. For Suzuki groups, this immediately follows from the description of $\Aut(P)$, where $P$ is a Sylow $p$-subgroup of $S$, in Proposition \ref{picardsuzuki}, but this holds in more generality when there are no outer diagonal or graph automorphisms. Recall that elements of $P$ have the form $x_{r_1}(u_1)...x_{r_n}(u_n)$, where $u_i$ are elements of the field, and $x_i$ roots. Then, since field automorphisms normalize $P$, every automorphism normalizing $P$ is generated by an element in the Borel subgroup of $S$ and a field automorphism. Now non-central elements don't normalize all root subgroups simultaneously, and the action of elements in the maximal torus $T$ on $P$ is given by
\[
{}^tx_r(u_r)=x_r(\theta_r(t)u_r)\text{ for $\theta_r\in X(\overline{T})$}.
\]
Since for a field automorphism $\sigma$, $u_r^{\sigma}\neq \theta_r(t)$, we conclude that $C_{\Aut(G)}(P)=Z(P)$, and, in particular, extensions by field automorphisms modify the inertial quotient when $S$ is ${}^2B_2(q)$ or ${}^2G_2(q)$. We still have to prove that we don't end up with another cyclic inertial quotient, but this can be readily checked. Suppose that for every $t\in T$, $(\theta_r(t)u_r)^{\sigma}=u_r^{\sigma}\theta_r(t)$, for all $u_r$. Then
\[
(u_r\theta_r(t))^{\sigma}=u_r^{\sigma}(\theta_r(t))^{\sigma}=u_r^{\sigma}\theta_r(t),
\]
so $\theta_r(t)$ must be in $\text{Fix}(\sigma)$ for all $t$, but this can't happen, unless one of $\sigma$ or $\theta_r$ is trivial.
Thus, for such $S$, $G=S$, and we conclude by Theorems \ref{picardsuzuki} and \ref{simplecases}.

If $S\in\left\lbrace PSL_2(q),PSU_3(q)\right\rbrace$, then as before field automorphisms won't yield a cyclic inertial quotient, but we have to deal with non-inner diagonal automorphisms. Let $S=PSL_2(q)$, then it is well known that outer diagonal automorphisms are given by conjugation with elements in $PGL_2(q)/PSL_2(q)$, so we can have $G=S$ or $G=PGL_2(q)$. In both these cases $\Pic(B)=\cT(B)$ by Theorem $\ref{simplecases}$ or an identical argument for $PGL_2(q)$. For $S=PSU_3(q)$, there exist outer diagonal automorphisms when $3$ divides $q+1$. In these cases we have to consider $PGU_3(q)$ as well, but analogously it holds $\Pic(B)=\cT(B)\simeq \Out(G)$.
\end{proof}
\begin{ack*}
I thank Charles Eaton for many useful discussions. Part of this work was done while I was visiting University of Florence, I therefore thank Silvio Dolfi and Eugenio Giannelli for hosting me.
\end{ack*}

\end{document}